\documentclass[12pt]{amsart}
\usepackage{txfonts}
\usepackage{mathbbold}
\usepackage{mathrsfs,pstricks,ifpdf,tikz}
\usepackage{amsfonts}
\usepackage{enumerate}
\usepackage{makecell}

\usepackage{latexsym}
\usepackage{amsmath,amssymb}
\usepackage{amsthm}
\usepackage{ifthen,calc}
\usepackage{color}
\usepackage{graphicx}
\usepackage{latexsym}
\usepackage{multirow}
\usepackage{diagbox}
\usepackage{float}
\usepackage{cases}
\usepackage[format=hang, margin=10pt]{caption}

\newtheorem{theorem}{Theorem}[section]

\newtheorem{corollary}[theorem]{Corollary}
\newtheorem{definition}[theorem]{Definition}

\newtheorem{lemma}[theorem]{Lemma}

\newcounter{hours}
\newcounter{minutes}
\newcommand{\printtime}{
    \setcounter{hours}{\time/60}%
    \setcounter{minutes}{\time-\value{hours}*60}
    \ifthenelse{\value{hours}<10}{0}{}\thehours:%
    \ifthenelse{\value{minutes}<10}{0}{}\theminutes}

\begin{document}

\title{Excluded checkerboard colourable ribbon graph minors}
\author{Xia Guo}
\address{School of Mathematical Sciences, Xiamen University, 361005, Xiamen, China}
\email{guoxia@stu.xmu.edu.cn}
\author{Xian'an Jin}
\address{School of Mathematical Sciences, Xiamen University, 361005, Xiamen, China}
\email{xajin@126.com}
\author{Qi Yan}
\address{School of Mathematical Sciences, Xiamen University, 361005, Xiamen, China}
\email{qiyanmath@163.com }
\thanks{This work is supported by NSFC (No. 11671336) and the Fundamental Research
Funds for the Central Universities (No. 20720190062). }

\noindent
\begin{abstract}
In this paper, we first introduce the notions of checkerboard colourable minors for ribbon graphs motivated by the Eulerian ribbon graph minors, and two kinds of bipartite minors for ribbon graphs, one of which is the dual of the checkerboard colourable minors and the other is motivated by the bipartite minors of abstract graphs. Then we give an excluded minor characterization of the class of checkerboard colourable ribbon graphs, bipartite ribbon graphs, plane checkerboard colourable ribbon graphs and plane bipartite ribbon graphs.
\end{abstract}

\keywords{ribbon graph; minor; checkerboard colourable; bipartite; plane graph.}

\subjclass[2000] {05C10; 05C45; 57M15}
\maketitle

\section{Introduction}
\noindent

The geometric dual is a fundamental concept in graph theory. It can be stated in the language of ribbon graphs as follows.
For a ribbon graph $G$, its geometrical dual $G^{\ast}$ is obtained by sewing discs, which will be the vertices of $G^{\ast}$, into the boundary components of $G$, and removing the interiors, which will be the faces of $G^{\ast}$, of all vertex discs of $G$.

To unify several Thistlethwaite's theorems, Chmutov \cite{C09} introduced the concept of partial duality which is a far-reaching extension of geometric duality.
Roughly speaking, the partial dual of a ribbon graph is obtained by forming the geometric dual with respect to an edge subset of this ribbon graph. The set of partial duals of a ribbon graph actually corresponds to the set of states of a link diagram. It includes the classical checkerboard graph and the Seifert graph as special cases and provides a bridge connecting knot theory and graph theory. In history the checkerboard graph was once used to solve the famous Tait conjecture \cite{M87} in knot theory. The partial duality has developed into a topic of independent interest (see \cite{EM13}) which has numerous applications in graph theory, knot theory, matroid theory and so on.

For an abstract graph, contracting a loop is the same as the deletion of this loop. But the ribbon graph has one more vertex after contracting an orientable loop.
Hence there is a fundamental difference between the theory of ribbon graph minors, defined by Moffatt \cite{M16}, and graph minors. Robertson and Seymour \cite{RS04} proved one can characterize every minor-closed family of graphs by a finite set of excluded minors (i.e. the Robertson-Seymour Theorem). For example, a graph is planar if and only if it contains no minors equivalent to $K_5$ or $K_{3,3}$.
In \cite{M16}, Moffatt gave a similar conjecture on the ribbon graph minors and an excluded graph minor characterization for the family of ribbon graphs representing knot and link diagrams.
He \cite{Mo16} further gave a characterization of ribbon graphs that the Euler genus of its partial duals is at most one.

It is well known that a 4-regular ribbon graph is the medial graph of some ribbon graph if and only if the 4-regular ribbon graph is checkerboard colorable. Clearly, a ribbon graph is checkerboard colourable if and only if its geometrical dual is bipartite. The set of Eulerian ribbon graphs properly include the set of checkerboard colourable ribbon graphs. A ribbon graph is Eulerian if and only if its geometrical dual is even-face. In \cite{MJ20}, Metsidik and Jin introduced the notions of Eulerian minor and even-face minor for ribbon graphs and characterized the Eulerian and even-face ribbon graphs by excluding these minors. In the case of abstract graphs, bipartite minors were introduced by Chudnovsky et. al. \cite{CK16} to characterize outerplanar graphs and forests. These works motivate us to study checkerboard colourable ribbon graph minors and bipartite ribbon graph minors.

The paper is organized as follows. We first give some preliminaries on ribbon graphs in Section 2. In Section $3$ we introduce the checkerboard colourable ribbon graph minor and its dual, the bipartite ribbon graph minor. We also give a different bipartite minor of ribbon graphs coming from bipartite minors of abstract graphs, we call it \emph{bipartite ribbon graph join minor}.
In Section $4$, we characterize checkerboard colourable ribbon graphs by checkerboard colourable ribbon graph minors and Eulerian ribbon graph minors, respectively. By the duality, we also obtain a characterization of bipartite ribbon graphs via bipartite ribbon graph minors and even-face ribbon graph minors. We further study the characterization of
bipartite ribbon graphs via bipartite ribbon graph join minors.
In the final Section $5$, we describe an excluded minors characterization of plane checkerboard colorable and plane bipartite ribbon graphs.

\section{Preliminaries on ribbon graphs}
\noindent

We begin with a brief review of ribbon graphs.

\begin{definition}\cite{BR02}
A ribbon graph $G$ is a (possibly non-orientable) surface with boundary, represented as the union of two sets of topological discs: a set $V(G)$ of vertices,  and a set $E(G)$ of edges such that
\begin{enumerate}
  \item the vertices and edges intersect in disjoint line segments;
  \item each such line segment lies on the boundary of precisely one vertex and precisely one edge;
  \item every edge contains exactly two such line segments.
\end{enumerate}

\end{definition}

Let $G=(V(G), E(G))$ be a ribbon graph. As in \cite{Me11}, we call the line segments intersected by vertices and edges \textit{common line segments}. By deleting the common line segments from the boundary of a vertex, we call the remaining line segments \textit{vertex line segments} of this vertex.  Similarly, if we delete the two common line segments from the boundary of an edge, we call the remaining two disjoint line segments \textit{edge line segments} of this edge,  shown in Figure \ref{f_1}, and they are coloured with red, black and blue respectively. Note that each edge of a ribbon graph consists of two half-edges.

\begin{figure}[htbp]
\centering
\includegraphics[width=1in]{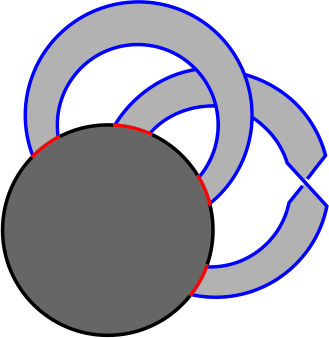}
\caption{A ribbon graph with red common line segments, black vertex line segments and blue edge line segments.}
\label{f_1}
\end{figure}

\begin{definition}\cite{C09}
An arrow presentation consists of a set of circles (corresponding to vertices) with pairs of labelled arrows (corresponding to edges), called  marking arrows, on them such that there are exactly two marking arrows of each label.
\end{definition}

Arrow presentations are equivalent to ribbon graphs, but they have advantage over ribbon graphs that they give a particularly efficient way to represent ribbon graphs. We emphasize that the circles in an arrow presentation are not equipped with any embedding in the plane or $\mathbb{R}^3$.

Let $G$ be an arrow presentation and $A\subseteq E(G)$.  The \textit{partial dual} $G^{A}$ of $G$ with respect to $A$ is constructed as follows. For each $e\in A$, suppose $e',~ e''$ are the two arrows labelled $e$ in the arrow presentation of $G$. Draw a line segment with an arrow on it directed from the head of $e'$ to the tail of $e''$ and from the head of $e''$ to the tail of $e'$, respectively. Label both of these arrows $e$, then delete $e',~e''$ and the arcs containing them, as shown in Figure \ref{f_2}. Note that $G^{E(G)}=G^{\ast}$. Clearly there is a natural 1-1 correspondence between the edges of $G$ and the edges of $G^{\ast}$. In particular we denote the corresponding edge of $e\in E(G)$ by $e^{\ast}$ in $G^{\ast}$.

\begin{figure}[htbp]
\centering
\includegraphics[width=3.0in]{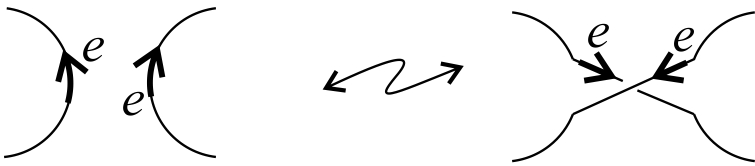}
\caption{Taking the partial dual respect to $e$ in an arrow presentation.}
\label{f_2}
\end{figure}

A \textit{deletion} $G-e$ of an edge from a ribbon graph $G$ is obtained from $G$ by removal of the edge ribbon $e$. It is equivalent to the deletion of the pair of marking arrows labeled by $e$ for an arrow presentation.
 Throughout this paper, we omit the set brackets in the case of a single element set.
 A \textit{contraction} $G/e$ is defined by the equation $G/e:= G^e-e$.

A ribbon graph is \textit{checkerboard colourable} if there exists an assignment with two colours to its boundary components such that the two edge line segments of each edges possess different colours.
A ribbon graph is said to be \textit{even-face} if each of its boundary components contains even number of edge line segments.
A ribbon graph is \textit{Eulerian} if the degree of each of its vertices is even. A ribbon graph is \textit{bipartite} if it, as an abstract graph, does not contain cycles of odd lengths.
A \textit{bouquet} denotes the ribbon graph with only one vertex. The \textit{Euler genus} $\gamma(G)$ of a ribbon graph $G$ is defined by $\gamma(G)=2c(G)-|V(G)|+|E(G)|-|F(G)|$, where $c(G)$, $|V(G)|$, $|E(G)|$ and $|F(G)|$ are the number of connected components, vertices, edges and faces (equivalently, boundary components) of $G$, respectively. In particular, $G$ is a plane graph if and only if $\gamma(G)=0$.

\section{Checkerboard colourable and bipartite minors of ribbon graphs}
\noindent

We first recall Eulerian ribbon graph minor and the even-face ribbon graph minor introduced in \cite{MJ20}.

The \textit{distance} between two vertex line segments (or edge line segments) lying on a boundary component is the minimum number of edge line segments lying between them on the boundary component.
The \textit{dual distance} between two vertex line segments (or common line segments) lying on a boundary of a vertex is the minimum number of common line segments lying between them on the boundary of the vertex.

If $e$ is not an orientable loop with the dual distance of common line segments of $e$ is odd, we call $G/e$ the \textit{proper edge contraction}.
The deletion $G-e$ is \textit{proper} if $G^{\ast}/e^{\ast}$ is proper. Suppose $C$ is a vertex boundary or a boundary component of a ribbon graph. If two arrows lie on $C$ and the directions of these two arrows are consistent with a direction of traveling around $C$, these two arrows are said to be consistent on $C$. Note that a common line segment of a ribbon graph $G$ corresponds to an edge line segment of $G^{\ast}$ and a vertex line segment of $G$ is still a vertex line segment of $G^{\ast}$.
Thus a vertex boundary of a ribbon graph $G$ corresponds to a boundary component of $G^{\ast}$. Equivalently, if $e$ is not the edge satisfying the following three conditions, then the deletion $G-e$ is proper.
\begin{enumerate}
\item the two edge line segments of $e$ lie on a same boundary component, denoted by $C_1$;
\item the distance of two edge line segments of $e$ is odd;
\item if we assign two arrows to the two edge line segments of $e$ such that these two arrows are consistent on the edge boundary of $e$, then these two arrows are consistent on $C_1$.
\end{enumerate}

The operation, \textit{evenly splitting a vertex}, is defined as follows. See Figure \ref{f 45}.
\begin{enumerate}
\item For a vertex, we pick two vertex line segments so that the dual distance of them is even;
\item Then we put two marking arrows of $e$ to these two vertex line segments such that they are consistent on this vertex boundary;
\item At last we contract the edge $e$.
\end{enumerate}
\begin{figure}[htbp]
\centering\includegraphics[width=4.0in]{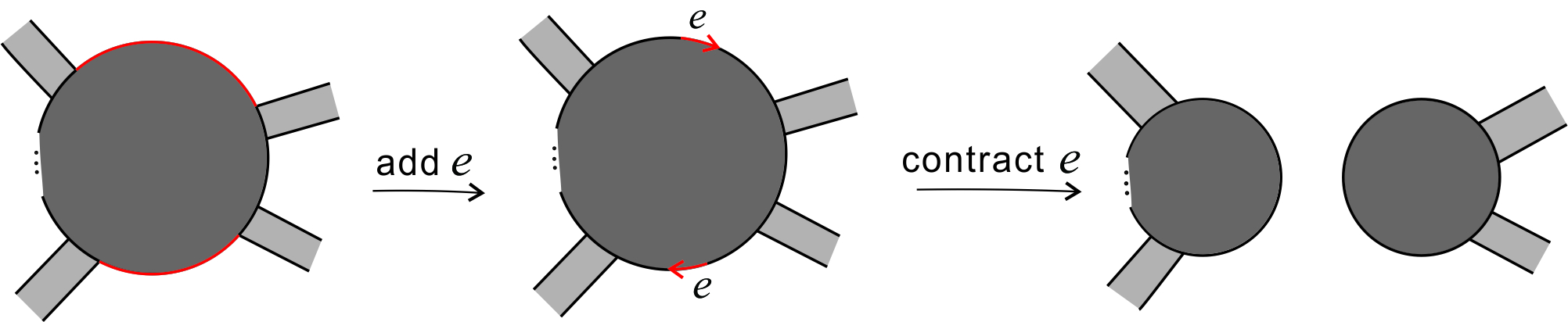}
\caption{Evenly splitting a vertex.}\label{f 45}
\end{figure}

Dually, the operation, \textit{evenly splitting a face}, is defined as follows. See Figure \ref{f 46}.

\begin{enumerate}
\item Choose two vertex line segments such that they lie on a boundary component and the distance of them is even;
\item Place two marking arrows of $e$ to this two vertex line segments such that they are consistent on this boundary component;
\item Contract the edge $e$ finally.
\end{enumerate}
\begin{figure}[htbp]
\centering\includegraphics[width=5.0in]{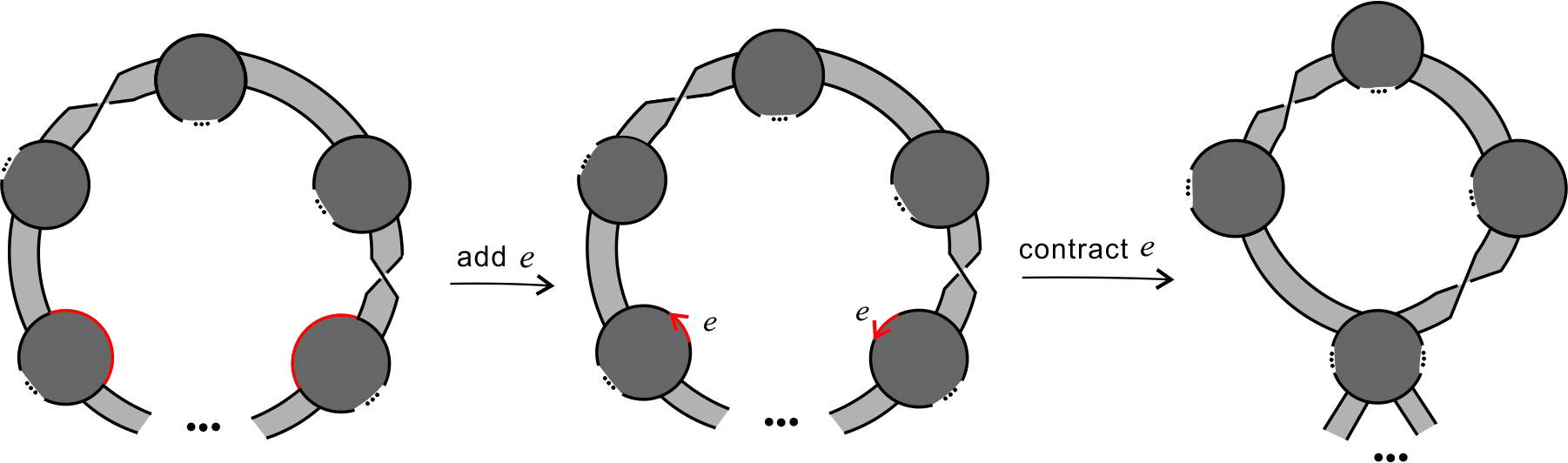}
\caption{Evenly splitting a face.}\label{f 46}
\end{figure}

Evenly splitting a face $f$ in an even-face ribbon graph $G$ corresponds to evenly splitting a vertex $f^{\ast}$ in the Eulerian ribbon graph $G^{\ast}$.

\begin{definition}\cite{MJ20}
\begin{enumerate}
\item A ribbon graph $H$ is an {\it Eulerian minor} of a ribbon graph $G$ if $H$ can be obtained from $G$  by a sequence of proper edge contractions, component deletions, or evenly splitting vertices.
\item A ribbon graph $H$ is an {\it even-face minor} of a ribbon graph $G$ if $H$ can be obtained from $G$ by a sequence of proper edge deletions, component deletions, or evenly splitting faces.
\end{enumerate}
\end{definition}

According to definitions, the following lemma is clear.

\begin{lemma}\cite{MJ20} \label{lemma2}
A ribbon graph $H$ is an Eulerian minor of a ribbon graph $G$ if and only if $H^{\ast}$ is an even-face minor of $G^{\ast}$.
\end{lemma}

Now we define the checkerboard colourable minor and bipartite minor for ribbon graphs by simply removing the ``proper" in the above definition. This is because contracting an edge preserves the checkerboard colorability and deleting an edge preserves the bipartite property.

\begin{definition}
A ribbon graph $H$ is a {\it checkerboard colourable minor} of a ribbon graph $G$ if $H$ can be obtained from $G$ by a sequence of edge contractions, component deletions, or evenly splitting vertices.
\end{definition}

Dually, we have

\begin{definition}
A ribbon graph $H$ is a {\it bipartite minor} of a ribbon graph $G$ if $H$ can be obtained from $G$ by a sequence of edge deletions, component deletions, or evenly splitting faces.
\end{definition}

The following lemma is clear.

\begin{lemma}\label{lemma4}
A ribbon graph $H$ is a checkerboard colourable minor of a ribbon graph $G$ if and only if $H^{\ast}$ is a bipartite ribbon graph minor of $G^{\ast}$.
\end{lemma}

Relations among Eulerian minors, even-face minors, checkerboard colourable minors and bipartite minors for ribbon graphs are shown in Figure \ref{f_6}.

\begin{figure}[htbp]
\centering
\includegraphics[width=3.5in]{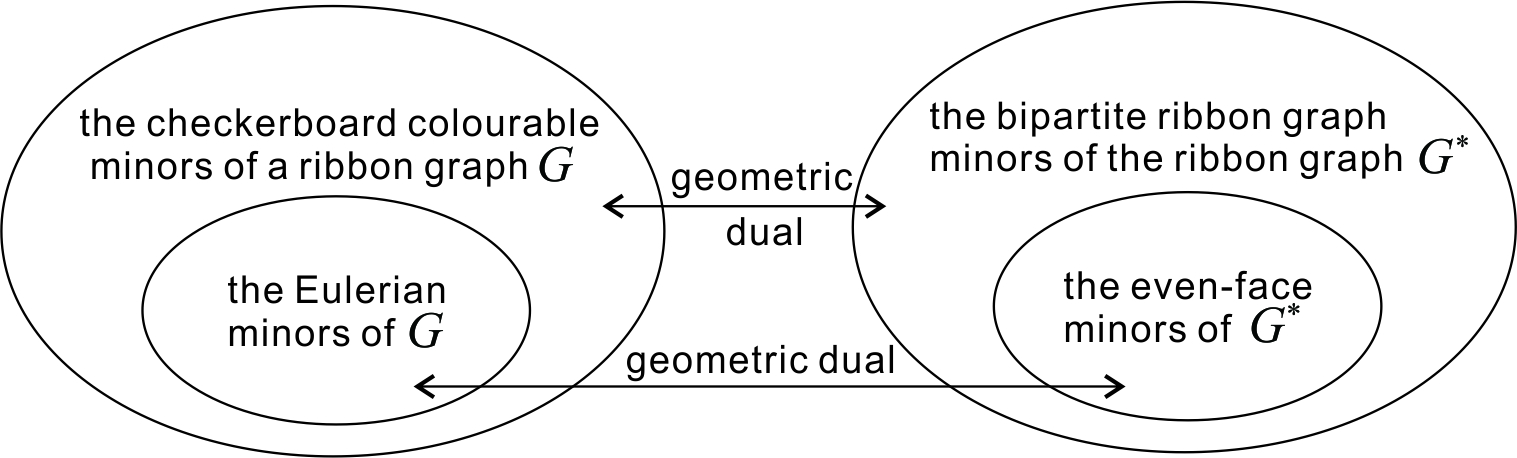}
\caption{Relations of the ribbon graph minors.}
\label{f_6}
\end{figure}

Let $G$ be a ribbon graph and $u$, $v$ be two vertices of $G$. The \textit{join} of $u$ and $v$ is merging vertex discs $u$ and $v$ together by identifying an arc lying on the vertex line segment of $u$ with an arc lying on the vertex line segment of $v$. Note that the way of joining two vertices is actually not unique, but it will have no any effect on the characterization of bipartite ribbon graphs later.
The join of two vertices is called \textit{permissible} if these two vertices have a common neighbor.

\begin{definition}
A ribbon graph $H$ is a {\it bipartite ribbon graph join minor} of a ribbon graph $G$ if $H$ can be obtained from $G$ by a sequence of permissible vertex joins, vertex deletions or edge deletions.
\end{definition}

Note that since each permissible join identifies two vertices that have a common neighbor, these two vertices must belong to the same part when the ribbon graph is bipartite. Hence, the set of bipartite ribbon graphs is bipartite ribbon graph join minor closed.

The notion of bipartite ribbon graph minor is different from bipartite ribbon graph join minor.
As shown in Figure \ref{f_5}, $H_1$ is a bipartite ribbon graph minor of the bouquet $B_3-e$ (see Figure \ref{f 40}), but $B_3-e$ contains no bipartite ribbon graph join minor equivalent to  $H_1$.
On the other hand, the bipartite ribbon graph join minor $H_2$ (having a unique boundary component) of the ribbon graph $G_2$ cannot be a bipartite ribbon graph minor of $G_2$.

\begin{figure}[htbp]
\centering
\includegraphics[width=5.5in]{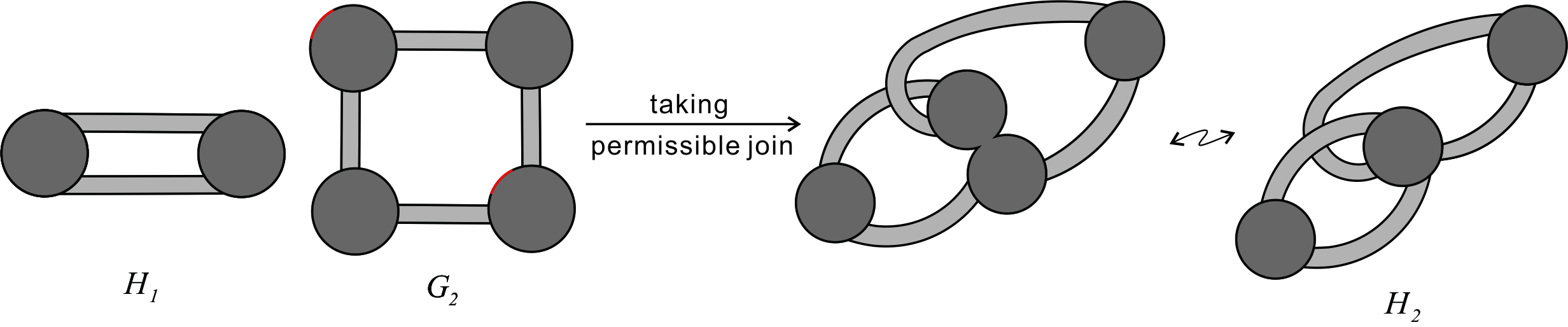}
\caption{Bipartite ribbon graph minor and bipartite ribbon graph join minor are different.}
\label{f_5}
\end{figure}

\section{Excluded minors for checkerboard colourable and bipartite ribbon graphs}

\begin{lemma}\label{lemma7}
The set of checkerboard colourable ribbon graphs is checkerboard colourable minor closed and Eulerian minor closed.
\end{lemma}
\begin{proof}
Let $G$ be a checkerboard colourable ribbon graph. Then every connected component of $G$ is checkerboard colourable, we obtain a checkerboard colourable ribbon graph when we delete a connected component.
The local presentation of ribbon graphs $G$ and $G/e$ is shown in Figure \ref{f_7}. Clearly, $G/e$ is also checkerboard colourable.
As shown in Figure \ref{f_8}, the ribbon graph obtained from $G$ by taking evenly splitting a vertex is also checkerboard colourable.
\begin{figure}[htbp]
\centering
\includegraphics[width=3.5in]{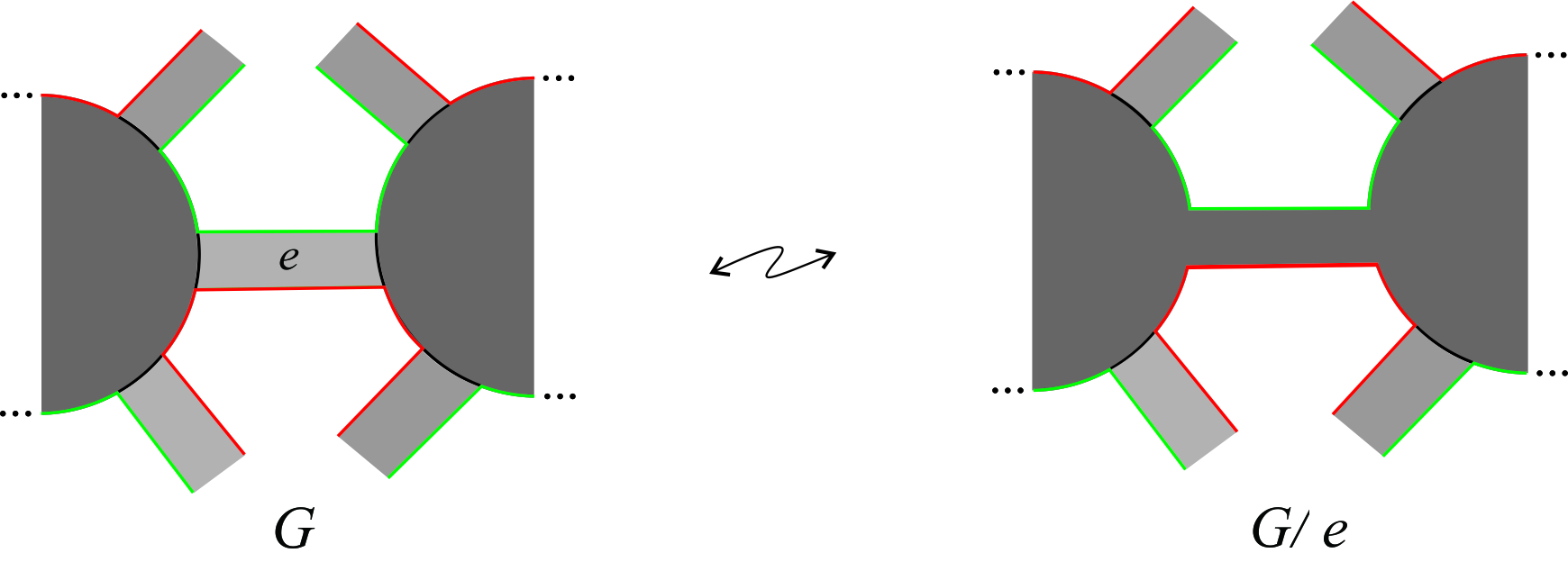}
\caption{The checkerboard coloured ribbon graphs $G$ and $G/e$.}
\label{f_7}
\end{figure}

\begin{figure}[htbp]
\centering
\includegraphics[width=5in]{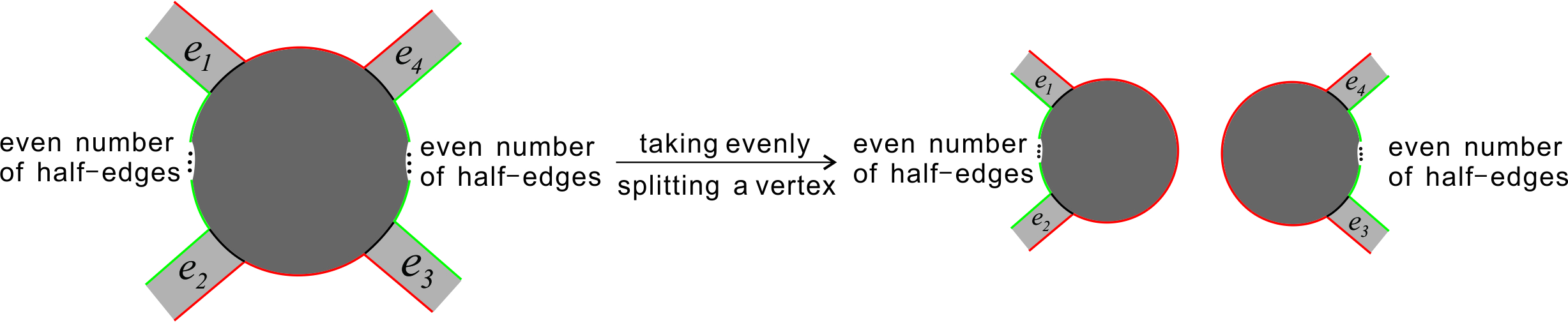}
\caption{Taking evenly splitting a vertex for a checkerboard coloured ribbon graph.}
\label{f_8}
\end{figure}
\end{proof}

By the dualtiy, we have

\begin{lemma}\label{lemma8}
The set of bipartite ribbon graphs is bipartite minor closed and even-face minor closed.
\end{lemma}

\begin{figure}[htbp]
\centering
\includegraphics[width=4.0in]{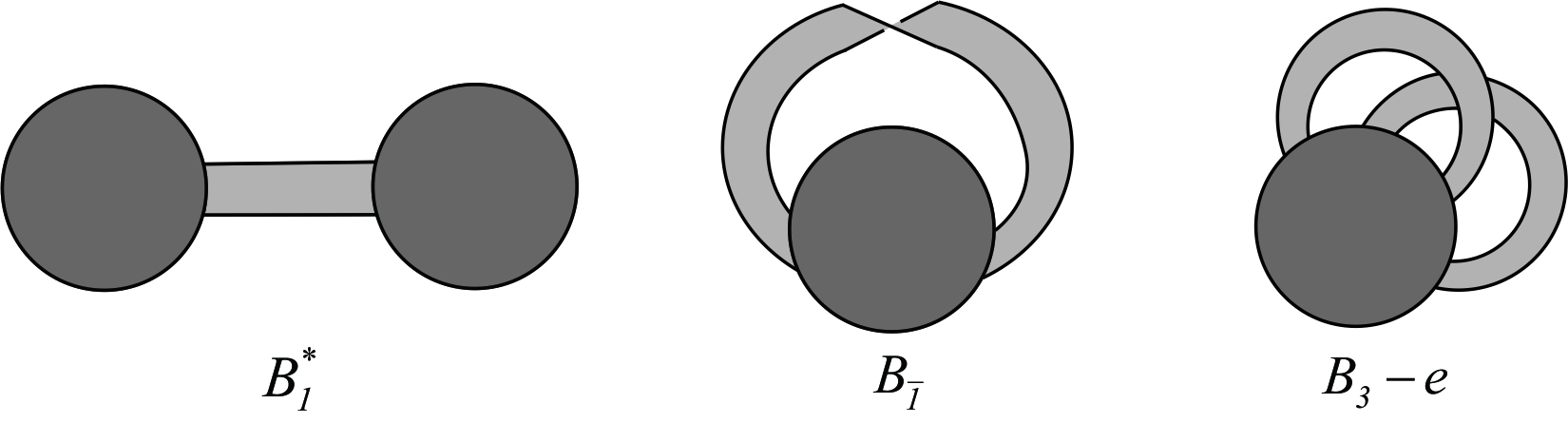}
\caption{The excluded Eulerian minors.}
\label{f 40}
\end{figure}
\begin{theorem}\label{theorem1}
A ribbon graph is checkerboard colourable if and only if it contains no Eulerian minor equivalent to $B_1^{\ast}$, $B_{\overline{1}}$ or $B_3-e$ as shown in Figure \ref{f 40}.
\end{theorem}
\begin{proof}
Obviously, $B_1^{\ast}$, $B_{\overline{1}}$ and $B_3-e$ cannot be checkerboard coloured.
By Lemma \ref{lemma7}, a checkerboard colourable ribbon graph contains no Eulerian minor equivalent to $B_1^{\ast}$, $B_{\overline{1}}$ or $B_3-e$.

On the contrary, suppose that the ribbon graph $G$ is not checkerboard colourable.
We delete all but a component which cannot be checkerboard colourable and we also denote the resulting ribbon graph $G$.

If $G$ is not an Eulerian graph, then it has at least two vertices of odd degree.
There exists a path, denoted by $P$, with the degrees of both two endpoints of $P$ are odd and the remainders are vertices of even degree.
Suppose $P=u_1u_2\dots u_k$.
By a sequence of proper edge contractions with respect to $E(P)-u_1u_2$, we obtain the ribbon graph $G/(E(P)-u_1u_2)$ with two adjacent  vertices of odd degree. By taking evenly splitting vertices with respect to these two adjacent vertices of odd degree, respectively, and deleting all the components but the component including these two vertices, we can get an Eulerian minor $B_1^{\ast}$ of $G$ .

If $G$ is an Eulerian graph, we claim that
if e is not a loop of $G$, then $G/e$ cannot be checkerboard colourable. Let $e=uv$.
We suppose $G/e$ is checkerboard coloured. The edge $e$ is an orientable loop in ribbon graph $G^e$.
The dual distance of common line segments of $e$ must be odd in $G^e$, otherwise,
the ribbon graph $(G^e)^e=G$ has vertices of odd degree, contradicting with $G$ is an Eulerian ribbon graph.
Thus, for ribbon graph $G/e$, the two vertex line segments which contain two common line segments of $e$ are coloured with two different colours. As shown in Figure \ref{f_7}, the ribbon graph $G$ is also checkerboard colourable, a contradiction.

We obtain a bouquet $G_1$ by a sequence of proper edge contractions with respect to the edges of a spanning tree of $G$.
It follows that the arrow presentation $G_1$ is not checkerboard colourable.
We assign two colours to the vertex line segments of the arrow presentation $G_1$ alternately.
Since converting an arrow presentation to a ribbon graph is adding a line segment from the head of one marking arrow to the tail of the other marking arrow for each pair of marking arrows,
hence the vertex line segments connected with the head of one marking arrow and that connected with the tail of the other marking arrow lie on the same boundary component.
Hence, $G_1$ is checkerboard colourable if and only if the colour of the vertex line segment connected with the head of one marking arrow  is the same as that connected with the tail of the other marking arrow for each pair of marking arrows of $G_1$.
Since $G_1$ is not checkerboard colourable, there exists one edge $f$ such that the colour of the vertex line segment connected with the head of one marking arrow of $f$ is different from that connected with the tail of the other marking arrow of $f$. Let $f'$, $f''$ be two marking arrows of $f$.

If the directions of $f'$ and $f''$ are consistent on the vertex boundary of $G_1$. Then there are odd number of marking arrows lying between $f'$ and $f''$. We proceed by induction on the total number of marking arrows lying between $f'$ and $f''$.
Let the number of marking arrows lying between $f'$ and $f''$ be $2k+1$ and $2l+1$ respectively, $k, l\geq 0$.
If $k=l=0$, we derive two Eulerian minors of $G_1$, shown in Figure \ref{f_10} (a) and (b).
Otherwise, by taking an evenly splitting vertex with respect to the marking arrows lying between $f'$ and $f''$, we obtain a ribbon graph with two vertices, denoted by $u_1$, $u_2$.
If $u_1$ and $u_2$ are adjacent, we contract a non-loop edge, then the resulting ribbon graph is a bouquet with both the number of marking arrows lying between $f'$ and $f''$ are odd and the addition of them is $2k+2l$.
If $u_1$ and $u_2$ are not adjacent, we delete the component except the component including the edge $f$. Then, for the resulting ribbon graph, both the number of marking arrows lying between $f'$ and $f''$ are odd and the addition of them is less than or equal to $2k+2l$.
Therefore, by induction, we get a bouquet with both the number of marking arrows lying between $f'$ and $f''$ are one, exhibited in Figure \ref{f_10} (a) and (b).
And, by contracting the edge $e$ in Figure \ref{f_10} (b), $B_{\overline{1}}$ is an Eulerian minor of Figure \ref{f_10} (b).

\begin{figure}[htbp]
\centering
\includegraphics[width=4.0in]{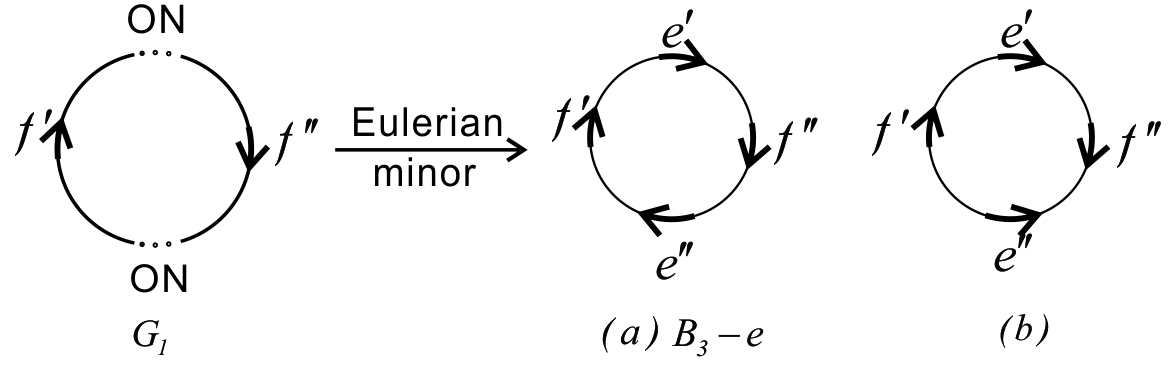}
\caption{The arrow presentation of $G_1$ when the directions of $f'$ and $f''$ are consistent and its Eulerian minors, in which ``ON" denotes the number of marking arrows is odd.}
\label{f_10}
\end{figure}

If the directions of $f'$ and $f''$ are inconsistent. Then there are even number of marking arrows lying between $f'$ and $f''$. By taking a sequence of evenly splitting vertices with respect to the marking arrows lying between $f'$ and $f''$ and component deletions, we have an Eulerian minor $B_{\overline{1}}$ of $G$, shown in Figure \ref{f_11}.

\begin{figure}[htbp]
\centering
\includegraphics[width=3.5in]{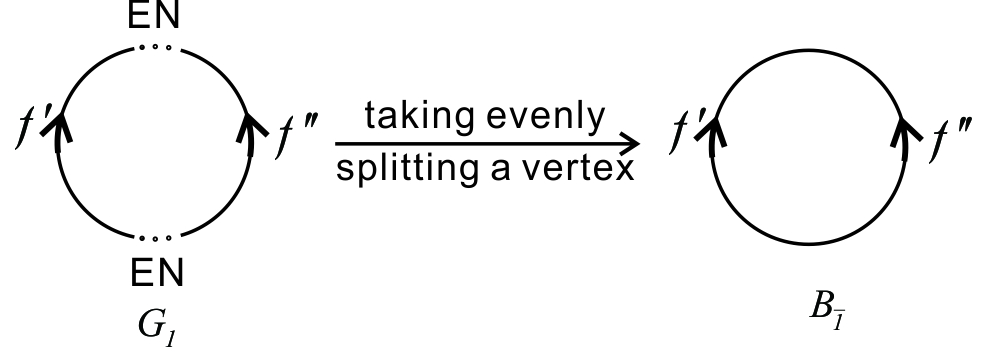}
\caption{The arrow presentation of $G_1$ when the directions of $f'$ and $f''$ are inconsistent and its Eulerian minors, in which ``EN" denotes the number of marking arrows is even.}
\label{f_11}
\end{figure}
\end{proof}

By using the checkerboard colourable minor, we obtain

\begin{theorem}\label{theorem2}
A ribbon graph is checkerboard colourable if and only if it contains no checkerboard colourable minor equivalent to $B^{\ast}_1$ or $B_{\overline{1}}$.
\end{theorem}
\begin{proof}
In the proof of Theorem \ref{theorem1}, we only contract the non-loop edges, so we have a ribbon graph is checkerboard colourable if and only if it contains no checkerboard colourable minor equivalent to $B_1^{\ast}$, $B_{\overline{1}}$ or $B_3-e$.
But, for the ribbon graph $B_3-e$, by contracting arbitrary one edge, we can obtain a checkerboard colourable minor $B_1^{\ast}$ of $B_3-e$. Therefore, the theorem is established.
\end{proof}

By the duality, we have

\begin{theorem}\label{theorem3}
A ribbon graph is bipartite if and only if it contains no even-face minor equivalent to $B_1$, $B_{\overline{1}}$ or $B_3-e$.
\end{theorem}

\begin{theorem}\label{theorem4}
A ribbon graph is bipartite if and only if it contains no bipartite ribbon graph minor equivalent to $B_1$ or $B_{\overline{1}}$.
\end{theorem}

To end this section, we give a characterization of bipartite ribbon graphs with excluded bipartite join minors for ribbon graphs.

\begin{theorem}\label{theorem5}
A ribbon graph is bipartite if and only if it contains no bipartite ribbon graph join minor equivalent to $B_1$ or $B_{\overline{1}}$.
\end{theorem}
\begin{proof}
It is clear that $B_{\overline{1}}$ and $B_1$ do not have a bipartition, so they cannot be bipartite ribbon graph join minors of bipartite ribbon graphs.
Conversely, if a ribbon graph $G$ is not bipartite, then there exists an odd cycle $C$ of $G$. We delete all edges but the edges of $C$. By taking a sequence permissible joins and deletions of vertex or edge, we can obtain a bipartite ribbon graph join minor $B_{\overline{1}}$ or $B_1$ of $G$.
\end{proof}

\section{Characterization of plane checkerboard colourable and plane bipartite ribbon graphs }
\noindent

\begin{lemma}\cite{M16}\label{lemma88}
If $H$ can be obtained from a ribbon graph $G$ by a sequence of edge contractions, vertex deletions and edge deletions, then $\gamma(H) \leq \gamma(G)$.
\end{lemma}

\begin{lemma}\cite{MJ20}\label{lemma89}
If $H$ is an Eulerian minor of a ribbon graph $G$, then $\gamma(H) \leq \gamma(G)$.
\end{lemma}

A loop at a vertex is trivial if there is no cycle or any other loop at $v$ which alternates with the loop.

\begin{theorem}\label{theorem6}
A checkerboard colourable ribbon graph is plane if and only if it contains no checkerboard colourable minor (or Eulerian minor) equivalent to $B_3$ or $B_{\overline{3}}-e$ as shown in Figure \ref{f_13}.
\end{theorem}
\begin{proof}
Since both the ribbon graphs $B_{\overline{3}}-e$ and $B_3$ are not planar, by Lemma \ref{lemma88} and Lemma \ref{lemma89}, they cannot be contained in checkerboard colourable minors (or Eulerian minors) of a plane ribbon graph.

Let $G$ be a checkerboard colourable ribbon graph.
On the contrary, we suppose $G$ is not planar. By contracting a spanning tree of $G$ we obtain an arrow presentation of a bouquet, denoted by $G_1$. When we contract a non-loop edge, we get a ribbon graph with the same number of connected components and faces, one less vertex and edge. Thus $\gamma(G_1)=\gamma(G)$.
Then we contract all trivial orientable loops of $G_1$, denoted by $G_2$.
Because by contracting a trivial orientable loop, we get a ribbon graph with one more connected component and vertex, one less edge, the same number of faces.
So $\gamma(G_2)=\gamma(G_1)$, that is $G_2$ is not planar.
We delete all components except one component that is not planar, denoted by $G_3$.
By Lemma \ref{lemma7}, the ribbon graph $G_3$ is also checkerboard colourable. Thus $G_3$ has no trivial loop (a checkerboard colourable bouquet has no trivial twisted loops).
Therefore, there exist two alternate edges with the dual distance of one common line segments of one edge and that of the other edge is zero, and denoted by $e_1$ and $e_2$.
Since $G_3$ is checkerboard colourable, as shown in Figure \ref{f_12}, the arrow presentation of $G_3$ has three cases.
Analogous to the Theorem \ref{theorem1}, by taking a sequence of evenly splitting vertices, non-loop edge contractions and component deletions, we can obtain the ribbon graphs represented in Figure \ref{f_13}.
For Figure\ref{f_13} (b), we gain a checkerboard colourable minor (or Eulerian minor) $B_{\overline{3}}-e$ by contracting a non-orientable loop.
It is in contradiction with the condition that $G$ contains no checkerboard colourable minor (or Eulerian minor) equivalent to $B_{\overline{3}}-e$ or $B_3$.
Hence $G$ is a plane ribbon graph.

\begin{figure}[htbp]
\centering
\includegraphics[width=4.0in]{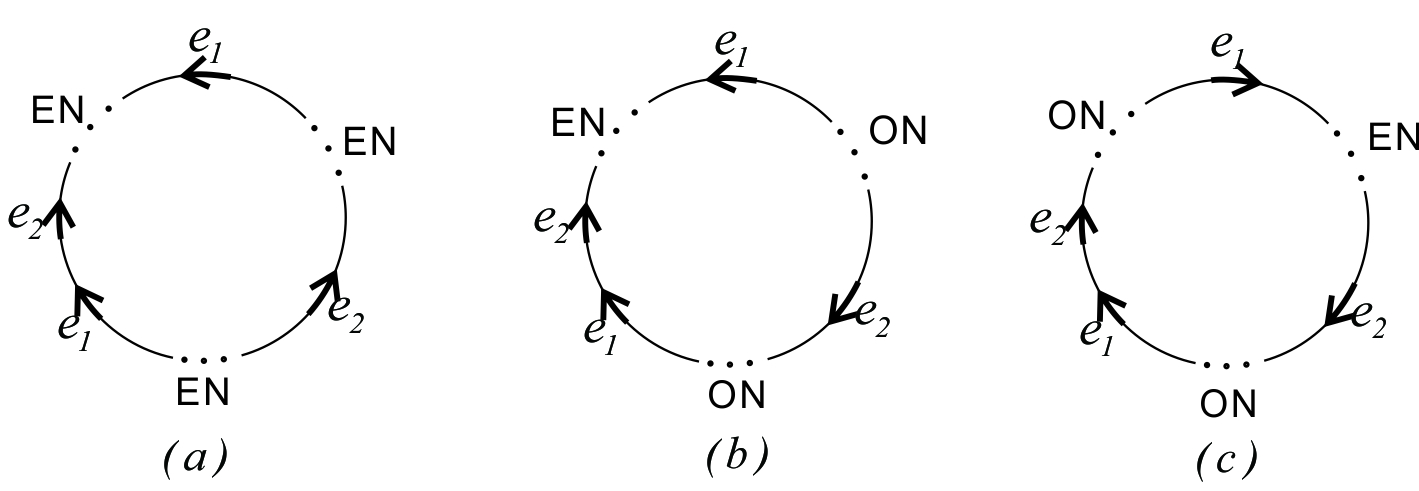}
\caption{The three possible cases of $G_3$.}
\label{f_12}
\end{figure}

\begin{figure}[htbp]
\centering
\includegraphics[width=4.0in]{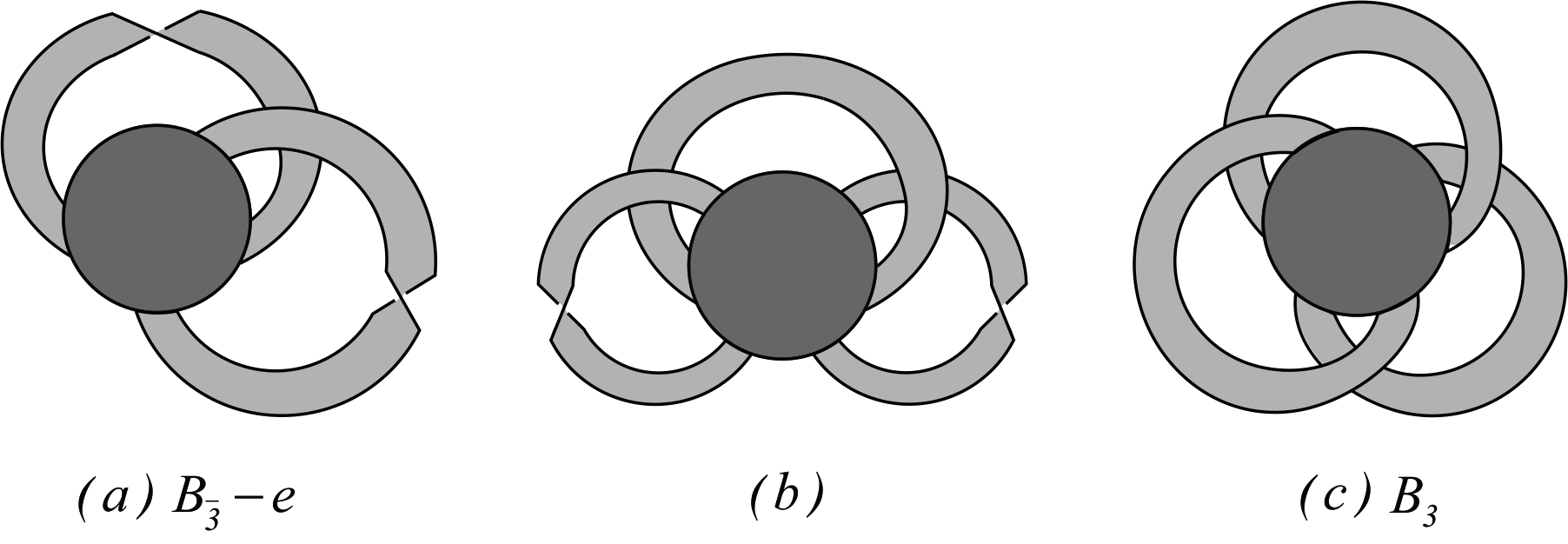}
\caption{Eulerian minors of Figure \ref{f_12}.}
\label{f_13}
\end{figure}
\end{proof}

By the duality, we have

\begin{theorem}\label{theorem7}
A bipartite ribbon graph is plane if and only if it contains no bipartite minor (or even-face minor) equivalent to $B^{\ast}_3$ or $(B_{\overline{3}}-e)^{\ast}$.
\end{theorem}

We have the following corollary.

\begin{corollary}
\begin{enumerate}
\item A ribbon graph is checkerboard colourable and plane if and only if it contains no Eulerian minor equivalent to $B^{\ast}_1$, $B_{\overline{1}}$,  $B_3$, $B_3-e$ or $B_{\overline{3}}-e$.
\item A ribbon graph is checkerboard colourable and plane if and only if it contains no checkerboard colourable minor equivalent to $B^{\ast}_1$, $B_{\overline{1}}$, $B_3$ or $B_{\overline{3}}-e$.
\item A ribbon graph is bipartite and plane if and only if it contains no even-face minor equivalent to $B_1$,$B_{\overline{1}}$, $B^{\ast}_3$, $B_3-e$ or $(B_{\overline{3}}-e)^{\ast}$.
\item A ribbon graph is bipartite and plane if and only if it contains no bipartite ribbon graph minor equivalent to $B_1$, $B_{\overline{1}}$, $B^{\ast}_3$ or $(B_{\overline{3}}-e)^{\ast}$.
\end{enumerate}
\end{corollary}

\end{document}